\documentclass[12pt]{article}
\usepackage{amsmath, amsfonts, amsthm, amssymb, color, hyperref, extarrows, enumitem, nicefrac,blindtext, mathtools, cite, anyfontsize}

\newtheorem{theorem}{Theorem}[section]

\newtheorem{cor}[theorem]{Corollary}
\newtheorem{lem}[theorem]{Lemma}

\numberwithin{equation}{section}

\newtheorem{example}{Example}

\usepackage[hyperpageref]{backref}

\hypersetup{
	colorlinks   = true,
	citecolor    = magenta}
	
\begin{document}
\title{\vspace{-1cm} \bf Continuous Sobolev functions with singularity on arbitrary real-analytic sets   \rm}
\author{Yifei Pan  \ \ and \ \  Yuan Zhang }
\date{}

\maketitle

\begin{abstract}
Near every point of a real-analytic set   in $\mathbb R^n$, we make use of Hironaka's resolution of singularity theorem to construct  a family of continuous functions in $  W^{1, 1}_{loc}$ such that their   weak derivatives  have (removable) singularity precisely on that set.   

\end{abstract}

\renewcommand{\thefootnote}{\fnsymbol{footnote}}
\footnotetext{\hspace*{-7mm}
\begin{tabular}{@{}r@{}p{16.5cm}@{}}
& 2020 Mathematics Subject Classification. Primary 32C07; Secondary 46E35, 14E15. \\
& Key words and phrases.
Logarithms, real-analyticity, Sobolev spaces, resolution of singularities. 
\end{tabular}}

\section{Introduction}
Given a domain $U$  in $\mathbb R^n, n\ge 1$, denote by  $W_{loc}^{k, p}(U)$ the Sobolev space  consisting of functions on $U$ whose $k$-th order weak derivatives exist and belong to $L_{loc}^p(U)$,  $k\in \mathbb Z^+, p\ge 1$.  In this note, we investigate a Sobolev property for the reciprocal of logarithms   of real-analytic functions near their zero sets.  In detail, given a    real-analytic  nonconstant function $ f   $ on $U$,     consider  \begin{equation}\label{v}
    v: = \frac{1}{\ln |f|}\ \ \text{on}\ \ U.
\end{equation} As we are solely interested in the Sobolev behavior of  $v$ near $f=0$, and additional singularity would be introduced near $|f|=1$, we further assume, say,  $|f|<\frac{1}{2}$ on $U$. Consequently $v$ is continuous on $U$. Moreover, letting $f^{-1}(0)$ be the  zero set  of $f$ in $U$, then  $v|_{f^{-1}(0)}  =0$, and $v$ is differentiable on  $U\setminus  f^{-1}(0) $. Note that $\text{codim}_{\mathbb R} f^{-1}(0) \ge 1  $ in general.

According to a classical result of Stein \cite[pp. 171]{St},   $\ln|f|\in \text{BMO}$ for any polynomial $f$. On the other hand,   a recent work\cite{SZ} by Shi and Zhang  showed that  for a real-analytic $f$ on $U$, if $ \text{codim}_{\mathbb R} f^{-1}(0)\ge 2$, then $\ln |f|\in W_{loc}^{1,1}(U)$. It is important to note that this codimension assumption  is essential and can not be dropped. In comparison to this result, although  $v$ in \eqref{v} exhibits   slightly greater regularity than $\ln|f| $, our first main theorem shows that $v$  belongs to $W_{loc}^{1,1}(U)$ regardless of the codimension  of $f^{-1}(0) $.


\begin{theorem}\label{main}
  Let $U$ be a domain in $\mathbb R^n, n\ge 1$. Let $f$ be a real-analytic nonconstant  function on $U$ and $|f|< \frac{1}{2} $ on $U$. The follow statements hold.\\
   1). $\frac{1}{\ln |f|}\in W^{1, 1}_{loc}(U)$. \\
   2). If $ \text{codim}_{\mathbb R} f^{-1}(0) = 1$, then $\frac{1}{\ln |f|}\notin W^{1, p}_{loc}(U) $ for any $p>1$.
\end{theorem}

  The main idea of the proof is to   use    the coarea formula to transform the integrals under consideration into new ones  along level sets of the function $f$. The $L^1$-integrability and the $L^p$-nonintegrability for $p>1$ that we seek  are thus consequences of certain quantitative properties of the level sets of $f$, which can be conveniently established by utilizing the powerful    Hironaka's resolution of singularity theorem and   the Lojasiewicz gradient inequality. A  novelty of Theorem \ref{main} is to provide ample $W^{1, 1}_{loc}$ functions. For instance, $\frac{1}{\ln|P(x)|}\in W^{1, 1}_{loc}$ for any polynomial $P$ near its zeros.  It is also interesting to point out that Theorem \ref{main} indicates that  Sobolev spaces in general do not satisfy  an openness property, in the sense that there exists a   class of functions in $W_{loc}^{k, p}(U)$ for some $p\ge 1$ but not in $W_{loc}^{k, q}(U) $ for any $q>p$.

 Unfortunately our method can not be applied directly in the smooth category, due to the absense of a Hironaka-type resolution property for smooth functions. It is natural to wonder if there is an easy way to verify the optimal Sobolev property of $v$, say, for any finitely vanishing smooth function $f$.  For instance, consider the function  $f(x, y): = y^2- \sin\left(e^{\frac{1}{x}}\pi\right)e^{-\frac{1}{x^2}}$, which  is smooth near $0\subset\mathbb R^2$ and vanishes to second order at $0$. It turns out, with a straight-forward computation, that  $\frac{1}{\ln |f|}\in W^{1,1} $ near $0$.   

 As a consequence of Theorem \ref{main},  the weak derivative  $\nabla v$ exists on $U$. Specifically, this implies that  the singularity set $f^{-1}(0)$ of $\nabla v$ in the classical sense   is actually a removable singularity in the weak sense. In other words, Theorem \ref{main} allows us to construct, for any  given real-analytic set, a continuous function in $W_{loc}^{1, 1}$ such that its weak derivative  have a removable  singularity precisely on that set.  

\begin{cor}\label{co}
    Let $A$ be a  real-analytic set  in $\mathbb R^n$. For every $p\in A$,  there exists an open neighborhood  $V$ of $p$  and a continuous function $u\in W^{1, 1}_{loc}(V)$, such that the set of  removable singularity of $\nabla u$ is $A\cap V$.
\end{cor}


Finally, we study the Sobolev property of $v$ in the special case when $f$ is  a holomorphic function on $U\subset \mathbb C^n$. Note that  in this case $\text{codim}_{\mathbb R} f^{-1}(0) = 2$ unless   $f\ne 0$ on $U$. 

\begin{theorem}\label{main2}
  Let $U$ be a domain in $\mathbb C^n$. Let $f$ be a holomorphic nonconstant function on $U$ and $|f|<\frac{1}{2}$ on $U$. The following statements hold.\\
  1). $  \frac{1}{\ln |f|}\in W^{1, 2}_{loc}(U).$\\
  2). If $  f^{-1}(0) \ne \emptyset$, then $  \frac{1}{\ln |f|}\notin W^{1, p}_{loc}(U)$ for any $p>2$.
\end{theorem}

\begin{cor}\label{co2}
    Let $A$ be a  complex analytic set  in $\mathbb R^n$. For every $p\in A$,  there exists an open neighborhood  $V$ of $p$  and a continuous function $u\in W^{1, 2}_{loc}(V)$, such that the set of  removable singularity of $\nabla u$ is $A\cap V$.
\end{cor}

In view of Theorem \ref{main} and Theorem \ref{main2}, it seems to have suggested a correlation  between    the codimension of the level sets and the Sobolev integrability index. Thus, one may ask whether $v\in W^{1, d}_{loc}(U)$ if $ \text{codim}_{\mathbb R} f^{-1}(0)=d $ for some  $0\le d\le n$. Unfortunately we do not have an answer to this question in general.

\section{Proof of Theorem \ref{main}}

Recall that the coarea formula states that,  given $\phi\in L^1(U)$, and a real-valued Lipschitz  function $f$ on $U$,      then
 \begin{equation}\label{ca}
     \int_U\phi(x) |\nabla f(x)|dV_x = \int_{-\infty}^\infty \int_{f^{-1}(t)}\phi(x) dS_x dt. 
 \end{equation} 
Here given $ t\in \mathbb R $,  $S_x$ is the $(n-1)$-dimensional Hausdorff measure of the level set $f^{-1}(t)$    of $f$ defined by 
$$ f^{-1}(t):= \{x\in U: f(x)=t\}.$$ 

Towards the proof of the main theorems, we shall fix the real-analytic (or holomorphic) function $f$  and use the following notation:  two quantities $A$ and $B$ are said to satisfy $A\lesssim B$, if  $A\le CB$ for some constant $C>0$ which  depends only on    $f$ under consideration. We say $A\gtrsim B$ if and only if $B\lesssim A$, and  $A\approx B$ if and only if $A\lesssim B$ and $B\lesssim A$ at the same time.


Given a set $A\subset \mathbb R^n$, denoting by $m(A) $   the Hausdorff measure. We first utilize  Hironaka's resolution  of singularity theorem to show the Hausdorff measure of level sets of real analytic functions is bounded (from above). This will be essential in proving a Harvey-Polking type removable singularity lemma for the weak derivatives of $v$.

\begin{theorem}\cite{At}\label{hi}
    Let $f$ be a real-analytic nonconstant function defined near a neighborhood of $0\in \mathbb R^n$. Then there exists an open set $U\subset \mathbb R^n$ near $0$,  a real-analytic manifold $\tilde U$ of dimension $n$ and a proper real-analytic map $\phi: \tilde U\rightarrow U$ such that\\
    1). $\phi: \tilde U\setminus \widetilde{f^{-1}(0)} \rightarrow U\setminus f^{-1}(0) $ is an isomorphism, where $\widetilde{f^{-1}(0)}: =  \{p\in \tilde U:  \phi(p)\in f^{-1}(0)\}$.\\
    2). For each $p\in \tilde U$, there exist    local real-analytic coordinates $(y_1, \ldots, y_n)$ centered at $p$,  such that  near $p$  one has 
    $$  f\circ \phi (y)= u(y)\cdot \Pi_{i=1}^n y_i^{k_i},$$
    where $u$ is real-analytic   and $u\ne 0$, $k_i\in \mathbb Z^+\cup\{0\}$. 
\end{theorem}
\medskip 

\begin{lem}\label{b1}
     Let $f$ be a  real-analytic nonconstant function on $U$. Then 
     \begin{equation*}
        m\left(f^{-1}(t)\right)\lesssim 1 \ \ \text{for all} \ \ |t|<<1. 
    \end{equation*}
\end{lem}

 \begin{proof}
Without loss of generality, assume $0\in U$ and $f(0)=0$.  Under the set-up of Hironaka's resolution Theorem \ref{hi}, for every $p\in \widetilde{f^{-1}(0)}   $, let  $(\tilde V, \psi)$ be  a coordinate chart near $p$ in $\tilde U$ such that  for $y\in \psi (\tilde V)\subset  \mathbb R^n$,
 $$f\circ\Phi(y): = f\circ \phi \circ \psi^{-1}(y) = u(y) \cdot \Pi_{i=1}^n y_i^{k_i}. $$
By properness of $\phi$,  $V: =\phi(\tilde V) $ is an open subset of $U$ near $p$.  Since $\phi$ is smooth on $\tilde U$, by shrinking $U$ if necessary, $\Phi: \psi(\tilde V)\rightarrow V  $ is smooth up to the boundary of $ \psi(\tilde V)$. By change of coordinates formula,
\begin{equation*}
    \begin{split}
      m\left(f^{-1}(t)\cap V\right) =\int_{\{f(x)=t\}\cap \phi(\tilde V)}   dS_{x} = \int_{\{f\circ \Phi(y)=t\}\cap \psi (\tilde V) }  \Phi^* dS_x \lesssim \int_{\{f\circ \Phi(y)=t\} \cap \psi (\tilde V)}   dS_y.
    \end{split}
\end{equation*}
Thus, in view of this and the fact that $u\ne 0$ on $\tilde U$, the proof boils down to show that \begin{equation}\label{in}
    m\left(A^n(t)\right)\lesssim 1  \ \ \text{for all}\ \  0<t<<1,
\end{equation}  where  
   \begin{equation}\label{an}
        A^n(t) =\{y\in \mathbb R^n:   \Pi_{i=1}^n y_i^{k_i}  =t, 0 <  y_i < 1, i=1, \ldots, n\}.
   \end{equation}  
Here the constant multiple for "$\lesssim$" in \eqref{in} is only dependent on $ k_i, i=1, \ldots, n$.   Clearly, one only needs to prove the case when all $k_i>0$. 
Let $k: = \sum_{i=1}^{n} k_i$.

We shall employ the mathematical induction on the dimension $n$ to prove \eqref{in} for all  level sets in the form of \eqref{an}. The $n=1$ case is trivial. Assume the $n=l$ case holds. Namely, for every level set $A^l(t)$ in $\mathbb R^l$ defined by \eqref{an}, $ m(A^l(t))\lesssim 1  $  for $0<t<<1$. When the dimension $n$ equals $l+1$, one first has $$  A^{l+1}(t)  \subset \cup_{j=1}^{l+1} A^{l+1}_j(t), $$ where for each $j = 1, \ldots, l+1$,
\begin{equation*}
    \begin{split}
       A^{l+1}_j(t): =\left\{y\in \mathbb R^{l+1}:   t^\frac{1}{k}< y_j< 1, 0<  y_i < 1  \ \text{if} \ i\ne j, \text{and}\  \Pi_{1\le i\le l+1, i\ne j} y_i^{k_i}  =ty_j^{-k_j}\right\}. 
    \end{split}
\end{equation*}  
 Further  denote $\hat y_j: = (y_1, \ldots, y_{j-1}, y_{j+1},\ldots, y_{l+1})\in \mathbb R^{l}$, $$ t': =ty_j^{-k_j},$$ and 
 $$ A^{l}_j(t'):  = \left\{\hat y_j\in \mathbb R^{l}: 0<  y_i < 1,  i\ne j,  \text{and}\  \Pi_{1\le i\le l+1, i\ne j}^l y_i^{k_i}  =t'\right\}.$$ Noting that $ t'<  t^{1-\frac{k_j}{k}}$ when $y_j>  t^\frac{1}{k}$,    we obtain 
 $$m\left(A^{l+1}(t) \right) \le \sum_{j=1}^{l+1}\int_{t^\frac{1}{k}}^1\int_{  \Pi_{1\le i\le l+1, i\ne j} y_i^{k_i}  =ty_j^{-k_j} }   dS_{\hat y_j} dy_j\le (1-t^\frac{1}{k}) \sum_{j=1}^{l+1}  \sup_{0<t'<t^{1-\frac{k_j}{k}}}m\left( A^{l}_j(t') \right). $$ 
   On the other hand, since $k_j<k$, one has $t^{1-\frac{k_j}{k}}<<1$ when $t<<1$. By the induction assumption and the fact that $A^{l}_j(t') $ is in $\mathbb R^l$,   
$$\sup_{0<t'<t^{1-\frac{k_j}{k}}} m\left(A^{l}_j(t') \right)\lesssim 1 \ \ \text{for all}\ \  0<t<<1.  $$
This finally gives 
$$ m\left(A^{l+1}(t)\right)\lesssim 1 \ \ \text{for all}\ \  0<t<<1. $$
The lemma is proved. 
 \end{proof}
\medskip

 \begin{lem} \label{lm1} Given a real-analytic nonconstant function $f$ on $U$ with  $|f|< \frac{1}{2} $ on $U$, let $v$ be defined in \eqref{v},  and \begin{equation}\label{g}
    g: = \frac{\nabla f}{f\cdot (\ln |f|)^2}\ \ \text{on}\ \ U.
\end{equation}  Then  $g\in L^1_{loc}(U)$.  Moreover,  one has
 $$ \nabla v = g  \ \ \text{on} \ \ U  $$  
 in the sense of distributions.
\end{lem}

\begin{proof}
 First, we show that $ g\in L^1_{loc}(U)$. Since $f$ is real-analytic on $U$,   shrinking $U$ if necessary, one can assume   $f$   to be (globally) Lipschitz on $U$.  Making use of the coarea formula \eqref{ca},   one gets
    \begin{equation*}
        \begin{split}
            \int_U|g(x)|dV_x &= \int_U  \frac{\left|\nabla f(x)\right| }{|f(x)| \left(\ln |f(x)|\right)^{2}}dV_x \\ 
            & \le \int_{-\frac{1}{2}}^{\frac{1}{2}}\int_{f^{-1}(t)}  \frac{1 }{|f(x)| \left(\ln |f(x)|\right)^{2}}dS_xdt  = \int_{-\frac{1}{2}}^{\frac{1}{2}}  \frac{ m( f ^{-1}(t)) }{|t|(\ln |t|)^2} dt.
        \end{split}
    \end{equation*}
Lemma \ref{b1} further allows us to infer
$$  \int_U|g(x)|dV_x \lesssim \int_0^{\frac{1}{2}}  \frac{ 1 }{t(\ln t)^2} dt = \int_{\ln 2}^\infty  \frac{ 1 }{s^2} ds<\infty.$$

Next, we show that given any testing function $\eta \in C_c^\infty(U)$,
  \begin{equation}\label{l3}
      -\int_U v\nabla \eta = \int_U \eta g.  
  \end{equation}
Since $v$ is differentiable away from $f^{-1}(0)$,  a direct computation gives 
  \begin{equation}\label{ea}
       \nabla v = g \ \ \text{on} \ \ U\setminus f^{-1}(0). 
  \end{equation}
In particular,   \eqref{l3} is trivially true if $K: = f^{-1}(0)\cap supp \ \eta = \emptyset $.

If $K\ne \emptyset$,    given  $\epsilon>0$ let 
$$K_\epsilon: = \{x\in U: dist(x, K)\le \epsilon\},$$ 
where $dist(x, K)$ is the distance function from $x$ to the set $K$. Let $\rho_\epsilon\in C^\infty(U)$ be  such that  
  $\rho_\epsilon =0$ in $K_\epsilon$, $\rho_\epsilon =1$ in $U\setminus K_{3\epsilon}$ and $|\nabla \rho_\epsilon|\lesssim \frac{1}{\epsilon}$ on $U$. See, for instance, \cite[Theorem 1.2.1-2]{Ho}. Then $\rho_\epsilon \eta\in C_c^\infty\left(U\setminus f^{-1}(0)\right)$. Using \eqref{ea}  we immediately have
  \begin{equation*}
        -\int_U v\nabla (\rho_\epsilon \eta) = \int_U \rho_\epsilon \eta g,
  \end{equation*}
  or  equivalently,
     \begin{equation}\label{l2}
        -\int_U v \eta  \nabla \rho_\epsilon - \int_U v\rho_\epsilon \nabla   \eta  = \int_U \rho_\epsilon \eta g.
  \end{equation}
We shall prove   
\begin{equation}\label{l5}
    \lim_{\epsilon \rightarrow 0}\int_U v \eta  \nabla \rho_\epsilon =0. 
\end{equation}
  If so, then passing $\epsilon\rightarrow 0$ in \eqref{l2}, we   obtain the desired equality \eqref{l3} as a consequence of Lebesgue's dominated  convergence theorem.

To prove \eqref{l5}, first by the assumption on $\rho_\epsilon$, 
\begin{equation}\label{bd}
        \left| \int_U v \eta  \nabla \rho_\epsilon \right|    =  \left| \int_{K_{3\epsilon\setminus K_\epsilon}} v \eta  \nabla \rho_\epsilon   \right|
    \lesssim  \frac{C}{\epsilon}\int_{K_{3\epsilon\setminus K_\epsilon}} |v|     
    \end{equation}
 for some constant $C$ dependent only on $\eta$. Since $f$ is Lipschitz on $U$, for any $x_0\in  f^{-1}(0)$, $|f(x)| = |f(x)-f(x_0)|\lesssim |x-x_0|$. In particular, $$|f(x)|\lesssim dist\left(x, f^{-1}(0)\right).$$ Thus   for all $x\in K_{3\epsilon}\setminus K_\epsilon$ (equivalently,  $\epsilon<dist\left(x, f^{-1}(0)\right)< 3\epsilon$),  one has
  $$ |v(x)|= \frac{1}{\left|\ln |f(x)|\right|}\lesssim \frac{1}{|\ln \left( dist\left(x, f^{-1}(0)\right) \right)  |}\approx \frac{1}{|\ln \epsilon|} $$
  for all $\epsilon$ small enough. 
Hence by \eqref{bd} 
\begin{equation}\label{l4}
 \left| \int_U v \eta  \nabla \rho_\epsilon    \right|\lesssim   \frac{C m(K_{3\epsilon} )}{\epsilon|\ln \epsilon|}.
\end{equation}
 On the other hand, according to a result of Loeser (see \cite[Theorem 1.1]{Lo86} and its consequent Remarks), 
 $$ m(K_{3\epsilon} )\lesssim \epsilon^{\text{codim}_\mathbb R f^{-1}(0)}\lesssim   \epsilon.$$
Here the last inequality has used the fact that   $   \text{codim}_{\mathbb R} f^{-1}(0)\ge 1 $ due to the  real-analyticity of $f$. The equality \eqref{l5} follows by combining the above with \eqref{l4}. 
\end{proof}

\medskip

\begin{proof}[Proof of Theorem \ref{main}:] Since $|f|<\frac{1}{2}$, we have $\left|\ln |f|\right|>\ln 2$ and so $|v|<\frac{1}{\ln 2}\in L^\infty(U)$. Part 1) follows from  this and Lemma \ref{lm1}. For part 2), we only need to show that  the function $g$ defined in \eqref{g} does not belong to $L_{loc}^p $ for any $ p>1$ near any neighborhood of $f^{-1}(0)$.  

First, according to the Lojasiewicz inequality, by shrinking $U$ if necessary, there exists some constant  $\beta\in (0, 1)$ such that 
\begin{equation}\label{lo}
    |\nabla f(x)| \gtrsim  |f(x)|^\beta, \ \ x\in U.
\end{equation}
 As a consequence of this,  
      \begin{equation*}
        \begin{split}
            \int_U|\nabla v(x)|^pdV_x &= \int_U  \frac{| \nabla f(x) | }{|\nabla f(x)|^{-(p-1)}|f(x)|^p  \left|\ln |f(x)|\right|^{2p}}dV_x\gtrsim \int_U  \frac{| \nabla f(x) |}{|f(x)|^{p - (p-1)\beta } \left|\ln |f(x)|\right|^{2p}}dV_x.               \end{split}
    \end{equation*}
  Utilizing  the coarea formula, we have for some $\epsilon_0>0$, 
  \begin{equation*}
        \begin{split}
            \int_U|\nabla v(x)|^pdV_x & \gtrsim  \int_{-\epsilon_0}^{\epsilon_0}\int_{f^{-1}(t)}  \frac{1 }{|f(x)|^{p - (p-1)\beta } \left|\ln |f(x)|\right|^{2p}}dS_xdt=   \int_{-\epsilon_0}^{\epsilon_0}   \frac{m(f^{-1}(t))  }{|t|^{p - (p-1)\beta } \left|\ln |t|\right|^{2p}} dt. 
        \end{split}
    \end{equation*}
    
    Since  $ \text{codim}_\mathbb R f^{-1}(0) =1$, there exists some $x_0\in f^{-1}(0)\cap U  $, such that  $|\nabla f(x_0)|\ne 0$. Let $V$ be a neighborhood of $x_0$ in $U$ such that  $|\nabla f|\gtrsim 1$ on $V$. Then for all   $t$ small enough, $m\left(f^{-1}(t)\cap V\right) \gtrsim 1 $. Consequently,  $m\left(f^{-1}(t)\right) \gtrsim 1$ for    $0<t<<1$.     Thus
    $$    \int_U|\nabla v(x)|^pdV_x\gtrsim  \int_{ 0}^{\epsilon_0}   \frac{1  }{t^{p - (p-1)\beta } \left|\ln t \right|^{2p}} dt. $$
    Note that    $p - (p-1)\beta>1 $ necessarily when $p>1$ and $\beta<1$. Hence the last term  is unbounded. The proof is complete. 
\end{proof}

\medskip
\begin{proof}[Proof of Corollary \ref{co}:]
Since $A$ is real-analytic, there exists an open neighborhood $V\subset \mathbb R^n$ of $p$ and  some real-analytic function $f$ on $V$ such that $A\cap V = \{x\in V: f(x) =0\}$. Then $u = \frac{1}{\ln |f|}$ is the desired function satisfying  the assumptions. 
\end{proof}

\medskip

 For functions (such as $\ln |f|$) with  singularities, its composition with another logarithm typically exhibits reduced singularities. The following theorem shows that composing extra logarithms    does not improve Sobolev regularity in general.

\begin{theorem}  Let $U$ be a domain in $\mathbb R^n, n\ge 1$. Let $f$ be a real-analytic nonconstant  function on $U$ and $|f|< \frac{1}{10} $ on $U$. The follow statements hold.\\ 
    1). $\frac{1}{\ln(\ln|f|)}\in W^{1, 1}_{loc}(U)$. \\
   2). If $ \text{codim}_\mathbb R f^{-1}(0)  =1$, then $\frac{1}{\ln(\ln|f|)}\notin W^{1, p}_{loc}(U) $ for any $p>1$.  
\end{theorem}

\begin{proof}
    Applying a similar approach as in the proof of Lemma \ref{lm1}, we first have 
    $$ \nabla \left(\frac{1}{\ln|\ln|f||}\right)  = \frac{\nabla f}{f\cdot \ln |f|\cdot (\ln|\ln|f||)^2 } \ \ \text{on} \ \ U  $$
  in the sense of distributions. Making use of the coarea formula and Lemma \ref{b1},     \begin{equation*}
        \begin{split}
            \int_U \left|\nabla \left(\frac{1}{\ln|\ln|f||}\right) \right|&\le  \int_{-\frac{1}{10}}^{\frac{1}{10}}\int_{ f^{-1}(t)}  \frac{1 }{|f(x)| |\ln |f(x)||\left(\ln |\ln 
      |f(x)||\right)^{2}}dS_xdt\\
            & \lesssim  \int_0^{\frac{1}{10}}  \frac{ 1 }{t|\ln t| (\ln|\ln t|)^2} dt  =   \int_{\ln 10}^\infty  \frac{ 1 }{t (\ln t|)^2} dt  =   \int_{\ln\ln 10}^\infty  \frac{ 1 }{t^2} dt\lesssim 1.
        \end{split}
    \end{equation*}

    In the case when $p>1$,  there exists  some  $0<\beta<1$ by \eqref{lo},  and some small $\epsilon_0>0$ such that
     \begin{equation*}
        \begin{split}
            \int_U \left|\nabla \left(\frac{1}{\ln|\ln|f||}\right)\right|^p&\gtrsim \int_U  \frac{| \nabla f(x) |}{|f(x)|^{p - (p-1)\beta }  |\ln |f(x)||^p\left(\ln |\ln 
      |f(x)||\right)^{2p}}dV_x\\ 
            & \gtrsim  \int_0^{\epsilon_0}  \frac{ 1 }{t^{p - (p-1)\beta }|\ln t| (\ln|\ln t|)^2} dt. 
        \end{split}
    \end{equation*}
    Since  $p - (p-1)\beta>1 $, the last term is divergent. This completes the proof of the theorem.  
\end{proof}

 \section{Proof of Theorem \ref{main2}}
To  prove Theorem \ref{main2} for holomorphic functions, we shall need the following well-known complex version Hironaka's resolution of singularity theorem. See, for instance, a lecture note \cite{Sm}.

\begin{theorem}\label{hi2}
    Let $f$ be a holomorphic function defined near a neighborhood of $0\in \mathbb C^n$. Then there exists an open set $U\subset \mathbb C^n$ near $0$,  a complex manifold $\tilde U$ of dimension $n$ and a proper holomorphic map $\phi: \tilde U\rightarrow U$ such that\\
    1). $\phi: \tilde U\setminus \widetilde{f^{-1}(0)} \rightarrow U\setminus f^{-1}(0) $ is a biholomorphism, where $\widetilde{f^{-1}(0)}: =  \{p\in \tilde U:  \phi(p)\in f^{-1}(0)\}$.\\
    2). For each $p\in \tilde U$, there exist    local holomorphic coordinates $(w_1, \ldots, w_n)$ centered at $p$,  such that  near $p$  one has 
    $$  f\circ \phi(w) = u(w)\cdot \Pi_{i=1}^n w_i^{k_i},$$
    where $u$ is holomorphic   and $u\ne 0$, $k_i\in \mathbb Z^+\cup \{0\}$. 
\end{theorem}
\medskip

\begin{proof}[Proof of Theorem \ref{main2}: ] 1). 
  Since $\bar\partial f =0$, and according to   Lemma  \ref{lm1},   
    $$\partial f    = \frac{\partial f}{f\cdot (\ln |f|)^2}\in L^1_{loc}(U)$$
  in the sense of distributions, we only need to show that   $$\frac{\partial f}{f\cdot (\ln |f|)^2}\in L^2_{loc}(U). $$  On the other hand, making use  of Hironaka's resolution of singularity Theorem \ref{hi2} for holomorphic functions, for every $  p\in \widetilde{f^{-1}(0)}   $, let  $(\tilde V, \psi)$ be  a coordinate chart near $  p$ in $\tilde U$ such that  for $w\in \psi (\tilde V)\subset \left \{w\in \mathbb C^n: |w_j|<\frac{1}{2}\right\}$,
 $$\tilde f(w): =  f\circ \phi \circ \psi^{-1}(w) = u(w) \cdot \Pi_{i=1}^n w_i^{k_i}, $$
 where $u\ne 0$ on $\psi (\tilde V) $ and $k_i\in \mathbb Z^+\cup\{0\}$. Let  $V: = \phi (\tilde V )$, $\Phi: = \phi \circ \psi^{-1}$, and $\text{Jac}_\Phi $ be the complex Jacobian of the holomorphic map $\Phi$. Note that the inverse matrix $(\text{Jac}_\Phi)^{-1} $ is smooth on $\psi\left(\tilde V\setminus \widetilde{f^{-1}(0)}\right) $, and $$\left|(\text{Jac}_{\Phi})^{-1}(w)   \cdot  \det(\text{Jac}_\Phi)(w) \right|\lesssim 1\ \ \ \text{for all}\ \ w\in \psi\left(\tilde V\setminus \widetilde{f^{-1}(0)}\right).$$  By change of variables formula,     \begin{equation*}
        \begin{split}
            \int_V  \frac{|\partial_z f(z)|^2}{|f(z)|^2(\ln |f(z)|)^4} dV_z&= \int_{\Phi^{-1}\left(V\setminus f^{-1}(0)\right)} \Phi^*\left(\frac{|\partial_z f(z)|^2}{|f(z)|^2(\ln |f(z)|)^4}  dV_z\right)\\
             & \lesssim \int_{\psi\left(\tilde V\setminus \widetilde{f^{-1}(0)}\right)}  \frac{ |\partial_w   \tilde f(w) |^2 |(\text{Jac}_{\Phi})^{-1}(w)|^2 }{|\tilde f(w)|^2(\ln |\tilde f(w)|)^4} |\det(\text{Jac}_\Phi(w))|^2 dV_w\\
            & \lesssim \int_{\psi(\tilde V)}  \frac{|\partial_w \tilde f(w)|^2}{|\tilde f(w)|^2(\ln |\tilde f(w)|)^4}  dV_w. 
        \end{split}
    \end{equation*}
Thus, the proof boils down to showing that for $j=1, \ldots, n$,  \begin{equation}\label{mm}
    \int_{\psi(\tilde V)}  \frac{|\partial_{w_j} \tilde f(w)|^2}{|\tilde f(w)|^2(\ln |\tilde f(w)|)^4}  dV_w\lesssim 1.
\end{equation}

For simplicity,  let $j =1$ in \eqref{mm}. If $k_1=0$, then $\partial_{w_1} \tilde f(w) = \partial_{w_1}   u(w)\cdot  \Pi_{i=1}^n w_i^{k_i}$. Since $ \frac{1}{(\ln |\tilde f(w)|)^4 } \lesssim 1$ and $u\ne 0$,  when $w$ is near $0$, 
$$   \frac{|\partial_{w_1} \tilde f(w)|^2}{|\tilde f(w)|^2(\ln |\tilde f(w)|)^4}  =  \frac{|\partial_{w_1}u(w)|^2}{|u(w)|^2(\ln |\tilde f(w)|)^4}\lesssim 1.$$
So \eqref{mm} holds.    If $k_1>0$, then   $\partial_{w_1} \tilde f(w) = \partial_{w_1}   u(w)\cdot  \Pi_{i=1}^n w_i^{k_i} +  k_1 u(w)\cdot w_1^{k_1-1}\cdot \Pi_{i=2}^n w_i^{k_i}. $ Hence 
$$   \frac{|\partial_{w_1} \tilde f(w)|^2}{|\tilde f(w)|^2(\ln |\tilde f(w)|)^4}  \lesssim   \frac{|\partial_{w_1}u(w)|^2}{|u(w)|^2(\ln |\tilde f(w)|)^4}  + \frac{k_1^2}{|w_1|^2(\ln |\tilde f(w)|)^4} \lesssim 1 + \frac{1}{|w_1|^2(\ln |\tilde f(w)|)^4} .$$
Note that when $w$ is close to $0$,   
\begin{equation}
  \left|  \ln |\tilde f(w)|\right| =  \left|  \ln |u(w)|  +\sum_{i=1}^n k_i\ln |w_i|\right|\gtrsim - \ln |w_1|.
\end{equation}
This leads to
\begin{equation*}
    \begin{split}
      \int_{\psi(\tilde V)}  \frac{|\partial_{w_1} \tilde f(w)|^2}{|\tilde f(w)|^2(\ln |\tilde f(w)|)^4}  dV_w&\lesssim  1+ \int_{\psi(\tilde V)} \frac{1}{|w_1|^2|\ln|w_1||^4}dV_w\\
      &\lesssim  1+ \int_0^\frac{1}{2} \frac{1}{s  (\ln s)^4}ds\lesssim 1.
    \end{split}
\end{equation*}
  \eqref{mm} and thus part 1) are proved.

2). Let $U_1$ be an open subset of $ U$ such that $f^{-1}(0)\cap U_1$ is regular. Then there exists a holomorphic coordinate change on $U_1$ such that under the new coordinates $(w_1, \ldots, w_n)$, one has $w_n = f(z)$. As a consequence of this,   
      \begin{equation*}
        \begin{split}
            \int_U\left|\frac{\partial_z f}{f\cdot (\ln |f|)^2}\right|^pdV_z &\ge  \int_{U_1}  \left|\frac{\partial_z  f}{f\cdot (\ln |f|)^2}\right|^p dV_z \approx  \int_{U_1}  \frac{1}{|w_n|^{p} \left|\ln |w_n|\right|^{2p}}dV_w\gtrsim  \int_0^{\epsilon_0} \frac{1}{s^{p-1}|\ln s|^{2p}}ds               \end{split}
    \end{equation*}
for some $\epsilon_0>0$.   Since $p> 2$,  the last term is unbounded. This proves part 2).  
  \end{proof}

\medskip

\begin{proof}[Proof of Corollary \ref{co2}:] The proof is similar to that of Corolary \ref{co}, with Theorem \ref{main} substituted by Theorem \ref{main2}, and is omitted.     
\end{proof}
\medskip

An  application  of Theorem \ref{main2} is to provide ample data to the $\bar\partial$ problem in complex analysis, in particular, within the framework of  H\"ormander's classical $L^2$ theory for   $\bar\partial$-closed forms with $L_{loc}^2$ coefficients. Normally,  generating smooth data  is straightforward. In the following, we construct data with singularity on  complex analytic varieties, where  H\"ormander's  theory can still be applied. 

\begin{example}
    Let $\Omega$ be a pseudoconvex domain in $\mathbb C^n$. 
    Let  $f$ be a holomorphic function on $\Omega$ such that $f^{-1}(0)\ne \emptyset$. 
    Given $\epsilon>0$, let
$$ U_\epsilon = \{z\in \Omega: |f|<\epsilon\}.$$
Choose $\chi\in C^\infty(\Omega)$ such that $\chi =1$ on $U_{\frac{1}{4}}$, and $\chi =0$ outside $U_{\frac{1}{2}}$. Then $g = \frac{\chi}{\ln |f|}\in W_{loc}^{1, 2}(\Omega)$ by Theorem \ref{main2}. Furthermore, $u: = \bar\partial g$ is a $\bar\partial$-closed $(0, 1)$ form with $L_{loc}^2$ coefficients with singularity at $f^{-1}(0)$. 
\end{example}

\bibliographystyle{alphaspecial}

\fontsize{11}{11}\selectfont

\vspace{0.7cm}

\noindent pan@pfw.edu,

\vspace{0.2 cm}

\noindent Department of Mathematical Sciences, Purdue University Fort Wayne, Fort Wayne, IN 46805-1499, USA.\\

\noindent zhangyu@pfw.edu,

\vspace{0.2 cm}

\noindent Department of Mathematical Sciences, Purdue University Fort Wayne, Fort Wayne, IN 46805-1499, USA.\\
\end{document}